\documentclass[11pt]{amsart}

\usepackage{times}
\usepackage{tikz}
\usepackage[left=1in,right=1in,top=1.5in,bottom=1.5in]{geometry}
\usepackage{amssymb}
\usepackage{latexsym, amsmath, amscd,amsthm,booktabs}
\usepackage{graphicx}
\usepackage[percent]{overpic}
\usepackage{units}
\usepackage{hyperref}
\usepackage{color}
\usepackage{xcolor}

\newcommand{\comment}[1]{}

\newtheorem{theorem}{Theorem}

\newtheorem{lemma}[theorem]{Lemma}
\newtheorem{proposition}[theorem]{Proposition}
\newtheorem{corollary}[theorem]{Corollary}

\theoremstyle{definition}
\newtheorem{definition}[theorem]{Definition}
\newtheorem{example}[theorem]{Example}

\newcommand{\Z}{\mathbb{Z}}

\newcommand{\cross}{\times}

\newcommand{\D}{\mathbb{D}}

\newcommand{\PD}{\mathbb{PD}}

\begin{document}
\title
     {Links and Planar Diagram Codes}
     
\author{Matt Mastin}
\address{Department of Mathematics,
Wake Forest University,
Winston-Salem, NC 27106}
\email{mastinjm@wfu.edu}

\begin{abstract} 
In this paper we formalize a combinatorial object for describing link diagrams called a Planar Diagram Code. PD-codes are used by the KnotTheory Mathematica package developed by Bar-Natan, et al. We present the set of PD-codes as a stand alone object and discuss its relationship with link diagrams. We give an explicit algorithm for reconstructing a knot diagram on a surface from a PD-code. We also discuss the intrinsic symmetries of PD-codes (i.e., invertibility and chirality). The moves analogous to the Reidemeister moves are also explored, and we show that the given set of PD-codes modulo these combinatorial Reidemeister moves is equivalent to classical link theory.
\end{abstract}
\date{\today}

\maketitle
\section[Introduction]{Introduction}

While the connection between Gauss codes and link diagrams has been widely examined in the literature (see for example chapter $17$ of \cite{GraphBook}), not much attention has been given to alternate combinatorial descriptions of link diagrams.  The goal of this paper is to formalize another combinatorial object called a Planar Diagram Code (PD-code) which can also represent link diagrams. We will take the approach of defining PD-codes as stand alone objects and then present their correspondence with link diagrams. This perspective has the advantage of producing a more general set of objects than the classical link diagrams. Similar observations have been made such as the ones by Kauffman \cite{VKnots1}, Nelson \cite{Nel1}, and Jablan, et al. \cite{JRR}. We will also give explicit algorithms for reconstructing the link diagram on a surface from a PD-code and discuss what information is lost when describing a diagram with a PD-code. A similar algorithm for reconstructing a diagram from a Gauss code has been discussed by Kauffman \cite{VKnots1}.

PD-codes and Gauss codes are closely related; each contains enough information to reconstruct a link diagram. However, PD-codes can be more useful for computation because the combinatorial information is more explicit. This is evidenced by the fact that the KnotTheory Mathematica package \cite{KnotAt} developed by Dror Bar-Natan and his students makes use of PD-codes for presenting link diagrams.

Section \ref{sect:Diags} will give the basic definitions of link diagrams and PD-codes, while section \ref{sect:Graphs} describes the correspondence between PD-codes and link diagrams. We will then discuss the analogues of the Reidemeister moves in section \ref{sect:RMoves} . In section \ref{sect:PDAct} we will discuss the intrinsic symmetries of PD-codes and then finish with a section on future directions.

\section[ Link Diagrams PD-Codes]{ Link Diagrams and PD-Codes}\label{sect:Diags}

We begin with our definition of an oriented and labeled link diagram. This definition is more or less standard, but we will be using the language of both graphs and link diagrams. The following will serve as our dictionary between these two viewpoints.

\begin{definition}\label{def:diagram}
Let $S$ be a smooth, closed, oriented, not necessarily connected, surface. A \textbf{labeled, oriented link diagram} on $S$ is a smooth oriented immersion $d:\sqcup_n S^1 \rightarrow S$ of $n$ disjoint circles into the surface with finitely many transverse double self-intersections and no other self-intersections. We take the circles $\sqcup_n S^1$ to be ordered. Call the self intersections of the immersion  \textbf{vertices} and the arcs between vertices \textbf{edges}. So, a diagram is in particular an embedded $4$-regular graph on $S$. A diagram also includes a labeling of the edges by pairs $(i,j)$ where $i$ is the index of the $S^1$ which contains the edge and the $j$'s give a cyclic ordering of the edges. So, a diagram is in particular a $4$-regular graph (a graph in which all vertices have degree $4$) on $S$ with with a preferred covering by oriented circuits. In the case of a knot ($n=1$) we will omit the first element in the pair.  The vertices will also be called \textbf{crossings} and are equipped with a coloring by the set $\{1,-1\}$ called the sign of the crossing. Two diagrams are equivalent if there is an isotopy of $S$ which brings one diagram to the other respecting the labeling of the edges and the signs of the crossings. We will denote the set of equivalence classes of diagrams by $\D$. A link diagram is \textbf{split} if one can embed a circle $\gamma$ in $S$, disjoint from the diagram, so that each connected component of $S \setminus \gamma$ contains part of the diagram.

Edges that are oriented toward the vertex will be called incoming edges and the others will be called outgoing. Each pair of non-adjacent edges will be referred to as either the \textbf{over-edges} or the \textbf{under-edges} as determined by the sign of the crossing as shown in Figure~\ref{fig:pushoff}. 
\end{definition}

It is a standard result that a link in a thickened surface $S \times [-1,1]$ can be recovered from a link diagram on $S$. The more interesting question of recovering a link from combinatorial data will be covered in the following section. But, first we will discuss how to encode the information contained in a link diagram in a PD-code.

\begin{figure}
\begin{center}
\begin{overpic}[height=5cm]{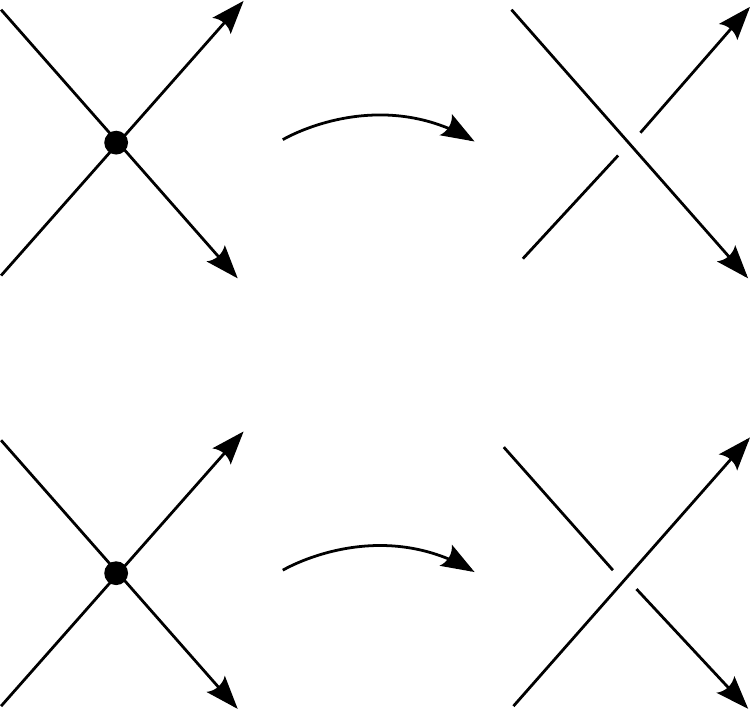}
\put(20,17){$-1$}
\put(20,75){$+1$}
\end{overpic}
\end{center}
\caption{\label{fig:pushoff} Over and under crossings. Here the orientation is taken to be the standard orientation of the plane of the page.}
\end{figure}

A PD-code is a set comprised of quadruples, where each entry in a quadruple may be viewed as an edge in a link diagram and each quadruple corresponds to a crossing. 

\begin{definition}\label{def:PD}
Given a link diagram on an oriented surface $S$, we generate the set of quadruples of the \textbf{PD-code} representing this diagram by the following procedure. For each crossing we include the quadruple of arc labels involved beginning with the incoming under-edge and proceeding around the crossings in the positively oriented direction of $S$ (see Figures~\ref{fig:PD} and \ref{fig:PDLink}). We give a positive sign to incoming edges and a negative sign to outgoing edges.

\begin{figure}
\begin{center}
\begin{overpic}[height=4cm]{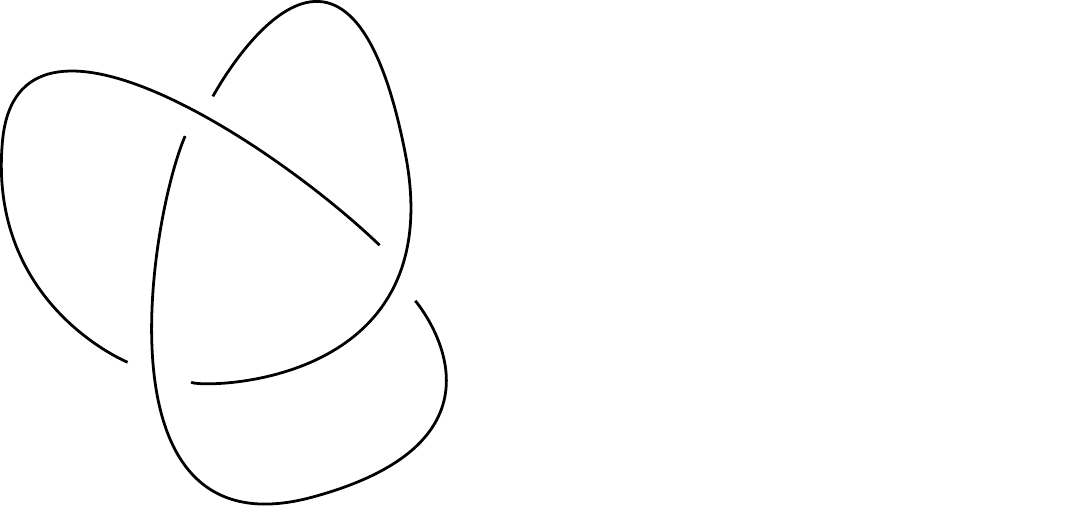}
\put(40,35){$1$}
\put(28,15){$2$}
\put(-3,30){$3$}
\put(28,33){$4$}
\put(35,-1){$5$}
\put(16,21){$6$}
\put(55,27){$\{ [+4,-2,-5,+1],[+2,-6,-3,+5],$}
\put(67,21){$[+6,-4,-1,+3]\}$}
\end{overpic}
\end{center}
\caption{\label{fig:PD} A diagram for $3_1$ and its PD-code. The labels are only single integers here as there is only one component. Note that we may omit directional arrows as the orientation can be inferred from the ordering of the edge labels.}
\end{figure}

\begin{figure}
\begin{center}
\begin{overpic}[height=7cm]{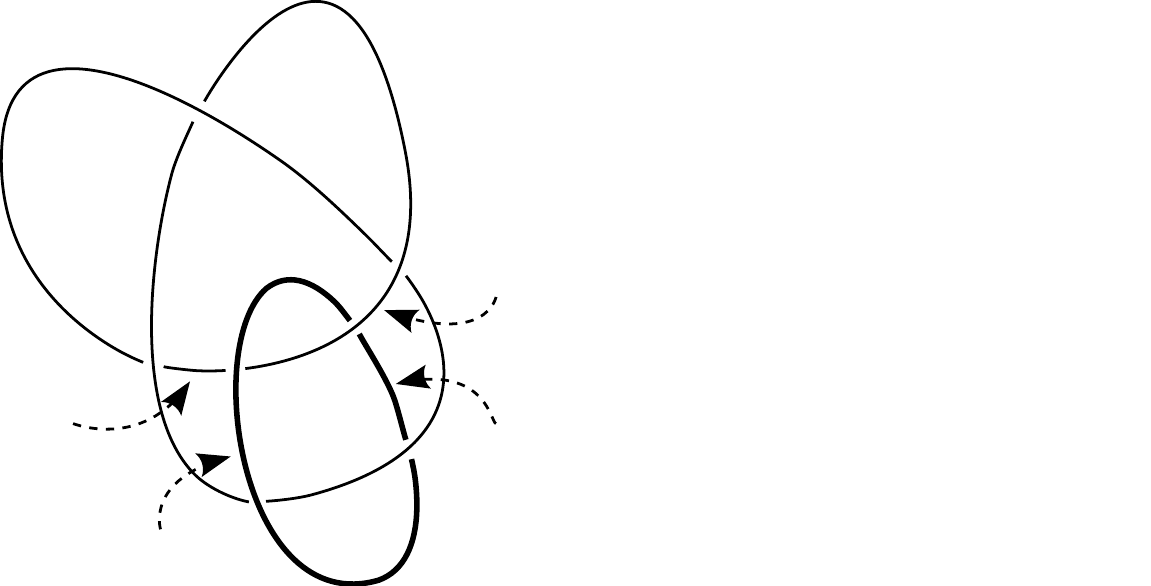}
\put(1,46){$(1,1)$}
\put(25,37){$(1,2)$}
\put(39,19){$(1,3)$}
\put(26,5){$(1,4)$}
\put(8.5,10){$(1,5)$}
\put(7,31){$(1,6)$}
\put(35,43){$(1,7)$}
\put(41,26){$(1,8)$}
\put(24,16){$(1,9)$}
\put(2,15){$(1,10)$}
\put(22,28){$(2,1)$}
\put(41,11){$(2,2)$}
\put(37,2){$(2,3)$}
\put(11,2){$(2,4)$}
\put(53,43){$\{[(1,+6),(1,-2),(1,-7),(1,+1)],$}
\put(55,38){$[(1,+2),(1,-8),(1,-3),(1,+7)],$}
\put(55,33){$[(2,+1),(1,-9),(2,-2),(1,+8)],$}
\put(55,28){$[(1,+9),(2,-1),(1,-10),(2,+4)],$}
\put(55,23){$[(1,+10),(1,-6),(1,-1),(1,+5)],$}
\put(55,18){$[(2,+2),(1,-4),(2,-3),(1,+3)],$}
\put(55,13){$[(1,+4),(2,-4),(1,-5),(2,+3)]\}$}
\end{overpic}
\end{center}
\caption{\label{fig:PDLink} A diagram of the link $7_7^2$ and the corresponding PD-code.}
\end{figure}
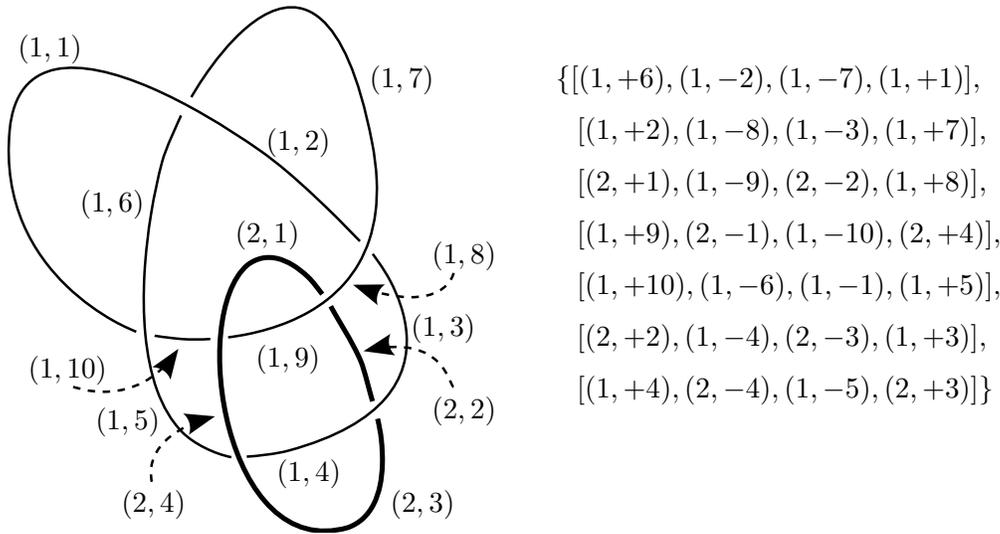


\end{definition}

In truth, the signs appearing in a PD-code are extraneous. They can be recovered by traversing the edges of each component in their ordering remembering that the first label in a quadruple is the incoming under crossing. However, keeping track of these signs will be convenient for our algorithm for constructing a link diagram on a surface. 

Given a link diagram we can generate a PD-code, but which collections of quadruples come from link diagrams? We will now define a particular class of codes and in the next section show that a link diagram on an orientable surface can be recovered from such codes.

\begin{definition}{\label{def:PDSet}}
Let $\overline{\PD}$ be the set of collections of quadruples of the labels 
$$\{(1,1),\ldots,(1,n_1),\ldots,(\mu,1),\ldots,(\mu,n_\mu)\}$$
satisfying the following properties.

\begin{enumerate}
\item Each edge label appears exactly twice, once positively (i.e., the sign on the second coordinate is positive) and once negatively (i.e., the sign on the second coordinate is negative). 
\item Each quadruple contains two positive edges and two negative edges and begins with a positive label.
\item The second coordinate of non-adjacent edge labels in each quadruple are consecutive modulo the number of arcs in those edges component. In addition, the first coordinate of non-adjacent edge labels in each quadruple are the same.
\item The first and third edge labels have opposite signs and the lesser edge label (in the ordering of the labels) is always positive. The second and fourth edge labels have opposite signs and the lesser edge label (in the ordering of the labels) is always positive.
\end{enumerate}

\end{definition}

\section[Graphs and Surfaces]{Graphs and Surfaces}\label{sect:Graphs}

In order to reconstruct a link from a PD-code we will first recover the underlying $4$-regular graph. This graph can be thought of as the shadow of a link diagram. Since it is the edges of the link diagram which are labeled it will be convenient to describe graphs by giving the edge set explicitly and describing the vertex set by edge adjacencies. Thus, we will use the following alternate definition of a graph.

\begin{definition}{\label{def:graph}}
A \textbf{graph} is an ordered pair $(E,V)$ where $E$ is a set of edges labeled by the pairs \\ $\{(1,1),\ldots,(1,n_1),\ldots,(\mu,1),\ldots,(\mu,n_\mu)\}$ and $V$ is a multi-set (a set with duplicates allowed, this is to allow loops at a single vertex) of unordered lists of labels such that each edge label appears exactly twice throughout the collection of lists. The unordered lists will be referred to as vertices and we say that an edge is incident to a vertex if it appears in the corresponding list. We say that two graphs $(E,V)$ and $(E',V')$ are isomorphic if there is a bijection $\phi:E \rightarrow E'$ such that $\{e_1,...,e_k\} \in V$ if and only if $\{\phi(e_1),...,\phi(e_k)\} \in V'$. Because we wish to recover a link diagram which has oriented edges, we will also consider graphs with oriented edges. When doing so, we will keep track of the orientations at the edge ends by labeling the incoming end with a positive label and the outgoing edge end with a negative label (this choice of sign convention is arbitrary, but should conform to the choice of signs of the labels in the PD-codes).
\end{definition}

\begin{example}{\label{ex:graph}}
Consider the graph with edge set $E=\{(1,1),(1,2),(1,3),(1,4),(1,5),(1,6)\}$ and vertex set \\ $V=\{ [4,2,5,1],[2,6,3,5],[6,4,1,3]\}$. Since only one label appears in the first coordinate of the edge labels we will, as mentioned in the previous section, drop the first coordinate and write the edge set as $\{1,2,3,4,5,6\}$. Graphs of this type will correspond to knots ($\mu=1$). Note that the graph with edge set $E'=\{1,2,3,4,5,6\}$ and vertex set $V'=\{ [2,4,1,5],[5,6,3,2],[1,4,6,3]\}$ is the same graph as $(E,V)$. 

We can now orient the edges of this graph by adding signs to the edge labels in the vertex set. For example, the oriented graph with edge set $E=\{1,2,3,4,5,6\}$ and vertex set \\$V=\ [+4,-2,-5,+1],[+2,-6,-3,+5],[+6,-4,-1,+3]\}$ gives the graph underlying the knot diagram in Figure \ref{fig:PD}.
\end{example}

%
%
%
%
%

It is also important to note that these are abstract graphs as opposed to particular embeddings. A diagram of an embedding of the graph from Example \ref{ex:graph} will be shown below. In order to specify a particular embedding of graph we must include information about the order in which the edges (in terms of the orientation of $S$) appear at each vertex.

\begin{definition}\label{def:cyclicorder}
Given a connected $4$-regular graph we define a \textbf{cyclic ordering at the vertices} to be an ordering of the four edge ends incident at vertex $v_i$. We will denote the ordering at vertex $v_i$ by $(e_i^0,e_i^1,e_i^2,e_i^3)$. Note that a single edge may be incident to the same vertex twice and thus the ordering labels need not correspond to distinct edges. 

\end{definition}

The reader may suspect that an arbitrary ordering of the edges may not produce a link diagram (and the reader would be correct), but as we will show later, graphs that are produced from PD-codes have an ordering on the edges that does produce a link diagram. Since each edge appears exactly twice (once positively and once negatively) in the cyclic orderings we can make the following definition.

\begin{definition}
We define the successor map by $s(e_i^j)=-e_i^{j+1 \mod(4)}$.
\end{definition}

The successor map gives the next edge to appear if we were to walk along an edge and turn right at each vertex. There is no particular reason for taking the convention of turning right and the theory would work just as well with the opposite choice. As another matter of convention, walking along a positive edge means we are going with the edge orientation and walking along a negative edge means we are going against the edge orientation. With this idea in mind we will generate a surface of identification using the successor map.

\begin{definition}\label{def:surface}
For each orbit of the successor map $s$ we construct an oriented polygon in the plane with sides given by the labels in each orbit where the labels are assigned clockwise around the polygon. The orientation of an edge agrees with the orientation inherited from the standard orientation of the plane for \emph{negative} labels and disagrees with this orientation for \emph{positive} labels (cf. Example \ref{ex:graphCell}). The orientations of the $2$-cells is inherited from the standard orientation on the plane. The orientations of the edges is induced by the orientation of the $2$-cells. This construction gives the \textbf{associated cell complex} to a given $4$-regular graph with a cyclic ordering at the vertices.
\end{definition}

\begin{example}{\label{ex:graphCell}}
Consider the graph with edge set $\{1,2,3,4,5,6\}$ and vertex set \\ $\{ [4,2,5,1],[2,6,3,5],[6,4,1,3]\}$ from Example \ref{ex:graph} . We can order the edges at each vertex and assign signs to distinguish incoming and outgoing edges as shown below.
$$\{ [+4,-2,-5,+1],[+2,-6,-3,+5],[+6,-4,-1,+3]\}$$
Note that this is precisely the PD-code of the example in Figure \ref{fig:PD} . The values of successor map are shown in Table \ref{fig:sOrbits}. Computing the orbits of this map on the set of edge labels gives that the set of faces is as follows. 

$$\{(+1,-4),(-1,-3,-5),(-2,+5),(+2,+6,+4),(-6,+3)\}$$

From Definition \ref{def:surface} we have two triangles and three bigons as shown in Figure \ref{fig:graphCells}.


\begin{figure}
\begin{center}
\begin{overpic}[height=8cm]{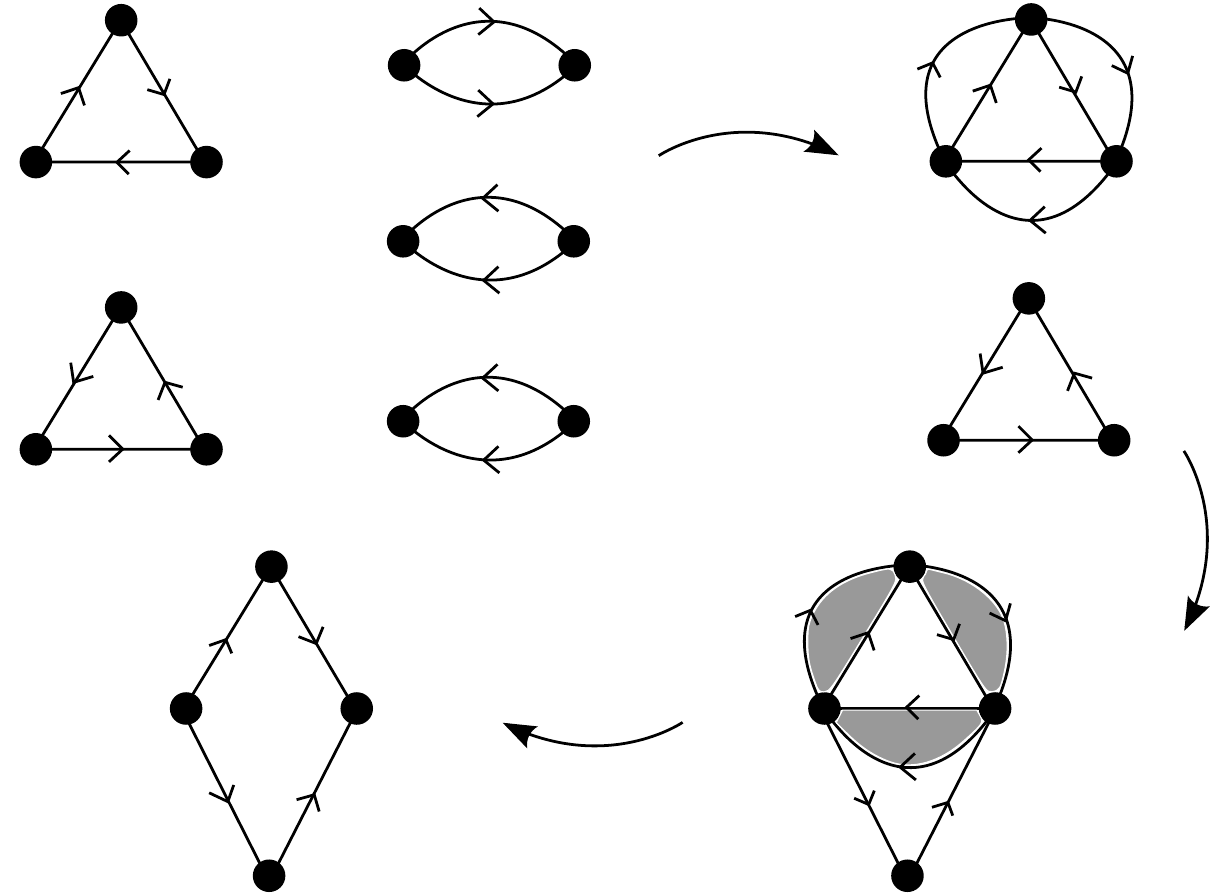}
\put(4,66){$2$}
\put(14,66){$6$}
\put(9,56){$4$}
\put(39,73){$3$}
\put(37,61){$6$}
\put(41,58){$4$}
\put(41,47){$1$}
\put(37,43){$2$}
\put(39,32){$5$}
\put(86,63){$6$}
\put(83,65){$2$}
\put(83,61){$4$}
\put(93,69){$3$}
\put(86,52){$1$}
\put(76,70){$5$}
\put(88,43){$3$}
\put(78,43){$1$}
\put(83,33){$5$}
\put(83,24){$3$}
\put(65,24){$5$}
\put(73,16){$4$}
\put(72,20){$2$}
\put(76,18){$6$}
\put(75,7){$1$}
\put(79,6){$3$}
\put(68,6){$5$}
\put(27,6){$3$}
\put(15,6){$5$}
\put(16,21){$5$}
\put(27,21){$3$}
\put(4,42){$1$}
\put(14,42){$3$}
\put(9,32.5){$5$}
\end{overpic}
\end{center}
\caption{\label{fig:graphCells} The faces and resulting polyhedron from Example~\ref{ex:graphCell}.} The shaded disks are contracted and interior edges removed in the last step to show that this is an identification for a $2$-sphere.
\end{figure}

\begin{table}[ht]
\begin{center}
\begin{tabular}{ccc}
\toprule
	
$s(+4)=+2$; & $s(+2)=+6$; & $s(+6)=+4$ \\	
$s(-2)=+5$; & $s(-6)=+3$; & $s(-4)=+1$ \\
$s(-5)=-1$; & $s(-3)=-5$; & $s(-1)=-3$ \\
$s(+1)=-4$; & $s(+5)=-2$; & $s(+3)=-6$ \\	

\bottomrule
\end{tabular}
\end{center}
\caption{\label{fig:sOrbits}The successor map from Example~\ref{ex:graphCell}}
\end{table}

\end{example}


\begin{proposition}\label{prop:GraphToSurface}
The cell complex associated to a connected $4$-regular graph with a cyclic ordering at the vertices is homeomorphic to a closed, connected, orientable surface.
\end{proposition}


\begin{proof}
Let $X$ be the cell complex corresponding to these identifications. To see that $X$ is homeomorphic to a closed surface we will first show that the gluing given in Definition~\ref{def:surface} produces a piecewise linear surface. There is nothing to check at interior points of the faces. Let $x$ be a point on edge $e$ and let $I$ be an open interval in $e$. There are exactly $2$ faces, say $f_1$ and $f_2$, incident to edge $e$ because the orbits partition the edges and each edge appears exactly twice. Each of these faces is a polygon in the plane so we may take $U_{f_1}$ and $U_{f_2}$ to be open balls in $f_1$ and $f_2$ which intersect the boundaries at exactly $I$. In $X$ we may now construct the set $U_{f_1} \cup I \cup U_{f_2}$ which is homeomorphic to an open disk in the plane. It is left to check that there is a neighborhood of each vertex which is homeomorphic to a ball in the plane. By construction each vertex is locally incident to exactly $4$ faces. In addition, near each vertex there are four distinct edges where adjacent edges in the quadruple bound a distinct triangle in the corresponding face (see Figure~\ref{fig:triangles}). Therefore, two edges either correspond to a unique triangle or they are not the boundary of any triangle. Thus, each vertex has a neighborhood which is homeomorphic to a piecewise linear disk. So, $X$ is homeomorphic to a closed piecewise linear surface. Since we have a finite cell complex we can appeal to the classification of surfaces and assume that we can smooth our surface. We will call this smooth surface the surface associated to the diagram. The original graph is embedded on the associated surface as a $1$-skeleton. Thus, since the graph is connected the surface must also be connected. 

To see that the surface is orientable we refer to the identification of polygons in the plane and compute the second integral cellular homology group. If we glue pairs of edges until we are left with a single polygon then we can remove edges and vertices until there is only a single $2$-cell remaining. Since we removed pairs of edges, the remaining edges must still appear in pairs once positively and once negatively. There is only a single $2$-cell, so the group of $2$-chains is generated by a single element. The boundary of this two-cell is the sum of the edges which is zero as the remaining edges pair off positively and negatively. This process is illustrated in Figure \ref{fig:graphCells}. So, the second integral cellular homology group is isomorphic to $\Z$ and the surface is therefore orientable.

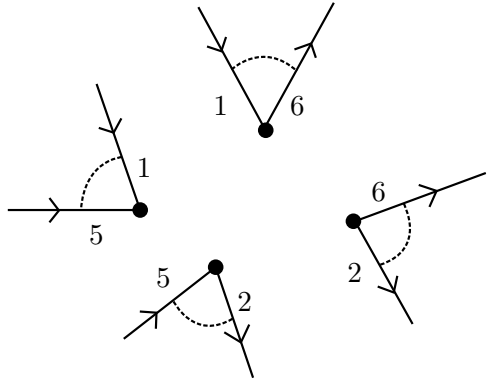
\begin{figure}
\begin{center}
\begin{overpic}[height=5cm]{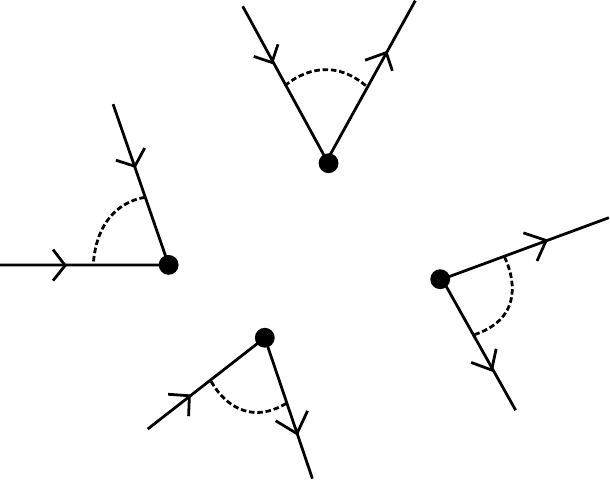}
\put(43,55){$1$}
\put(59,55){$6$}
\put(76,37){$6$}
\put(71,20){$2$}
\put(48,14){$2$}
\put(31,19){$5$}
\put(17,28){$5$}
\put(27,42){$1$}
\end{overpic}
\end{center}
\caption{\label{fig:triangles} Four faces coming together at a vertex to form a disk.}
\end{figure}

\end{proof}

If one starts with a disconnected graph, then the resulting surface is the disjoint union of the connected surfaces resulting from each connected component of the graph.


\begin{proposition}\label{prop:PDtoD}
Given a PD-code we can produce a well-defined link diagram on an orientable surface $S$.
\end{proposition}

\begin{proof}
Given a PD code $C$ we first produce a graph with a cyclic ordering at the vertices. Let the vertex set be the set of quadruples and the edge set be the set $\{(1,1),\ldots,(1,n_1),\ldots, (\mu,1),\ldots,(\mu,n_\mu)\}$ of labels in $C$ (neglecting signs). An edge is incident to a vertex if its label appears in the vertices' quadruple. We get an ordering at the vertices by the order of the labels in each quadruple. Thus, by Proposition~\ref{prop:GraphToSurface} we can produce an orientable, not necessarily connected, surface $S$ with a cell decomposition corresponding to $C$. We will now use the remaining data in $C$ to define a link diagram on $S$ using the $1$ skeleton of the cell decomposition. Beginning with edge $(1,1)$ we traverse the $1$ skeleton by choosing the non-adjacent edge according to the ordering at each vertex. Property $4$ of Definition~\ref{def:PDSet} implies that the resulting cycle can be given an orientation by orienting each edge from the negative label to the positive. Thus, we have an oriented immersion from $S^1$ into $S$. We repeat this process until every edge is contained in a cycle thus producing an oriented immersion $d:\sqcup_n S^1 \rightarrow S$ where the $\sqcup_n S^1$ are ordered. We color a vertex with $1$ if the corresponding quadruple in the PD-code is of the form $[+\alpha, -\gamma, -\beta, +\delta]$ and we color a vertex with $-1$ if the corresponding quadruple in the PD-code is of the form $[+\alpha, +\gamma, -\beta, -\delta]$. 
\end{proof}

It is important to note that the map of Proposition~\ref{prop:PDtoD} and the algorithm for computing a PD-code from a diagram are not inverses. For example, if one placed the standard diagram for the trefoil inside an embedded disk on the torus one would obtain the same PD-code as the standard diagram for the trefoil on a sphere. In this case, the map of Proposition~\ref{prop:PDtoD} will produce the diagram of the trefoil on the sphere (cf. Example~\ref{ex:trefoil}).

\begin{definition}\label{def:PDSurface}
The \textbf{surface associated to a PD-code} $C$ is the closed, orientable surface associated to the $4$-regular graph with an ordering at the vertices underlying $C$ (cf. Proposition~\ref{prop:PDtoD}). We say the PD-code is \textbf{connected} if the associated surface is connected. 
\end{definition}

Proposition~\ref{prop:PDtoD} and Proposition~\ref{prop:GraphToSurface} imply that the surface associated to a PD-code is well-defined.

\begin{definition}
Let the \textbf{Euler characteristic of a PD-code} be the Euler characteristic of the corresponding polyhedron. We define the \textbf{genus of a PD-code} to be $g=\sum_i g_i$ where the sum is over the connected components of the associated surface.
\end{definition}

\begin{theorem}\label{theorem:euler}
The Euler characteristic of an $n$ crossing PD-code $C$ is given by $N(s)-n$ where $N(s)$ is the number of orbits of $s$ on the set of edge labels.
\end{theorem}

\begin{proof}
The polyhedron corresponding to $C$ has $n$ vertices, $2n$ edges, and $N(s)$ faces. Computing the Euler characteristic we see that $n-2n+N(s)=N(s)-n$.
\end{proof}

\begin{example}
\label{ex:trefoil}
Consider again the PD-code $\{ [+4,-2,-5,+1],[+2,-6,-3,+5],[+6,-4,-1,+3]\}$ from Figure~\ref{fig:PD}. We saw in Example \ref{ex:graphCell} that the orbits of the successor map are as follows.

$$\{(+1,-4),(-1,-3,-5),(-2,+5),(+2,+6,+4),(-6,+3)\}$$ 

\vspace{0.5cm}

Note that there are three vertices and five $2$-cells. So, by Theorem \ref{theorem:euler} we have $\chi = 5-3 = 2$. Thus, the surface associated with this PD-code is a single $2$-sphere. It is important to note that we have implicitly used the fact that the surface is connected and orientable to conclude that having an Euler characteristic of $2$ implies that the surface is a $2$-sphere.

%
%



\end{example}

As previously mentioned, if we start with a link diagram, generate its PD-code and then produce a link diagram from the PD-code we will not necessarily end up with the diagram that we started with. For example, two link diagrams on a torus that differ by a Dehn twist would have the same PD-code. However, for the case of classical non-split link diagrams (i.e. non-split diagrams on a $2$-sphere) these operations are in fact inverses. We now make this precise.

\begin{definition}\label{def:PDSET}
Define $\PD$ to be the subset of the set of all PD-codes $\overline{\PD}$ (cf. Definition~\ref{def:PDSet}) whose associated surfaces are single $2$-spheres.
\end{definition}

\begin{proposition}\label{prop:PDisD}
$\PD$ is in bijection with the set of non-split link diagrams on $2$-spheres (cf. Definition~\ref{def:diagram}).
\end{proposition}

\begin{proof}
Definition~\ref{def:PD} gives a map from link diagrams to PD-codes and Proposition~\ref{prop:PDtoD} gives a map from PD-codes to link diagrams. We must show that these maps are inverses when restricted to non-split links diagrams on the $2$-sphere and the set of PD-codes $\PD$ (cf. Definition~\ref{def:PDSET}). 

First, consider the image $C$ of a link diagram $D$ under the map of Definition~\ref{def:PD}. The link diagram $D$ on the $2$-sphere gives a natural cell decomposition. The $0$-cells are the self-intersections, the $1$-cells are the arcs, and the $2$-cells are polygons which form the complement of the diagram. It is clear that the map of Proposition~\ref{prop:PDtoD} recovers a cell decomposition with the same number of $0$-cells and $1$-cells and by the definition of the successor map the $2$-cells are polygons in bijection with the decomposition given by the link diagram. Moreover, the gluing map given in Proposition~\ref{prop:PDtoD} ensures that the $2$-cell adjacency relations in the link diagram as well as the associated surface agree. Therefore, the $1$-skeleton of $D$ and the $1$-skeleton of the diagram reconstructed from $C$ are isomorphic graphs. It is a corollary of the Jordan-Schonflies theorem that two embeddings of a graph with the same cyclic order at the vertices on a surface are related by an orientation preserving homeomorphism of the surface (cf. \cite{GraphSurface} Theorem $3.2.4$). \comment{The cyclic orderings of these two graphs are the same, thus this homeomorphism is orientation preserving and therefore isotopic to the identity on $S^2$.} Any orientation preserving homeomorphism of $S^2$ is isotopic to the identity. Thus, the $1$-skeletons of $D$ and $C$ are isotopic and hence equivalent as link diagrams.

Now consider the image $D$ of a PD-code $C$ under the map given in Proposition~\ref{prop:PDtoD}. The arc labels of $D$ are given by the labels of $C$ as is the over/under crossing information. We again have that the adjacency relations of the $2$-cells are preserved. Therefore if we read the PD-code as in Definition~\ref{def:PD} from $D$ we recover $C$.
\end{proof}

\begin{example}
We now finish our running example with the PD-code of Example \ref{ex:trefoil} and show that it in fact recovers the knot diagram from Figure \ref{fig:PD} . In Figure \ref{fig:graph_diagram} we see the graph from Example \ref{ex:trefoil} along with the edge labels and vertex coloring as specified in Proposition \ref{prop:PDtoD} . To see that we have recovered the link diagram that we started with we can simply replace the vertices with the correct crossing configuration as shown in Figure \ref{fig:pushoff} and check that the resulting diagram agrees with Figure \ref{fig:PD}.

\begin{figure}
\begin{center}
\begin{overpic}[width=11cm]{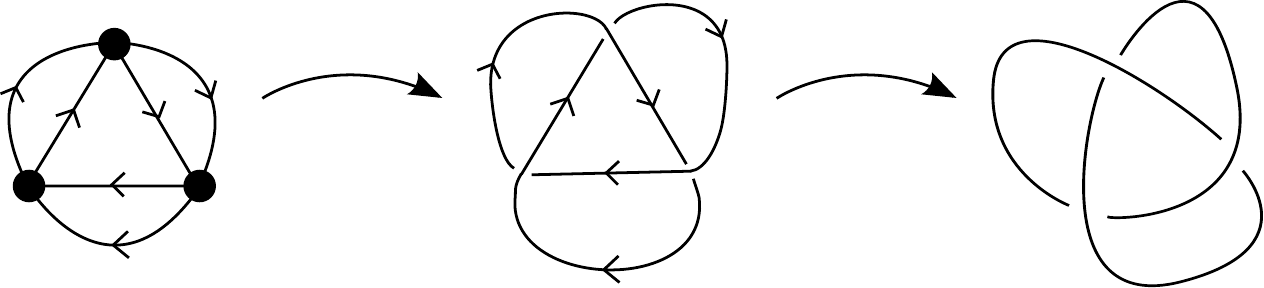}
\put(2,19){$5$}
\put(14,19){$3$}
\put(8,-1){$1$}
\put(7,13){$2$}
\put(11,10){$6$}
\put(6,9){$4$}

\put(41,22){$5$}
\put(55.5,22){$3$}
\put(47,-2.5){$1$}
\put(46,14){$2$}
\put(50,11){$6$}
\put(45,10){$4$}

\put(77,18){$3$}
\put(87,11){$6$}
\put(98,17){$1$}
\put(92,11){$4$}
\put(90,7){$2$}
\put(97,-1){$5$}
\end{overpic}
\end{center}
\caption{\label{fig:graph_diagram} Here we see the graph from Example \ref{ex:trefoil} as it is embedded on $S^2$. If we redraw the crossings as shown in Figure \ref{fig:pushoff} we see that the diagram of Figure \ref{fig:PD} is recovered. The second step shown is simply a rotation in the plane of the page.} 
\end{figure}
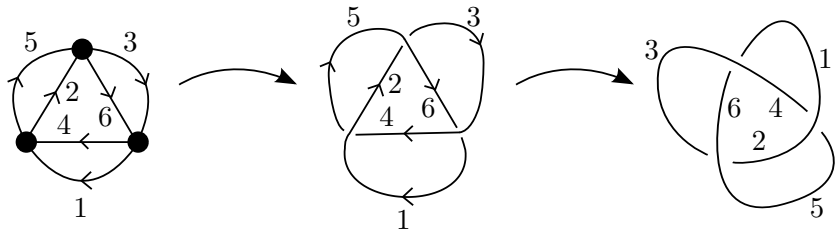

\end{example}

\begin{example}\label{ex:torus_example}
We end this section with an example of a PD-code which represents a link diagram on a torus. Consider the PD-code $$\{[+5,+2,-6,-3],[+3,-1,-4,+6],[+1,+4,-2,-5]\} $$ obtained from the link diagram shown in Figure \ref{fig:torus_diagram} by choosing an incoming under crossing and sign at each vertex. The orbits of the successor map are $$\{(+1,-4,-6,+3), (+2,+6,-3,-5,-1,+4),(+5,-2)\} .$$ At this point we can compute the Euler characteristic of the surface using Theorem \ref{theorem:euler}. There are $3$ orbits of the successor map and $3$ crossings, so we have $\chi = 3 - 3 =0$. Thus, the surface is a torus. The surface of identification is given in Figure \ref{fig:torus_example}. The diagram in this example is that of a virtual knot with $3$ classical crossings and $2$ virtual crossings.

\begin{figure}
\begin{center}
\begin{overpic}[width=7cm]{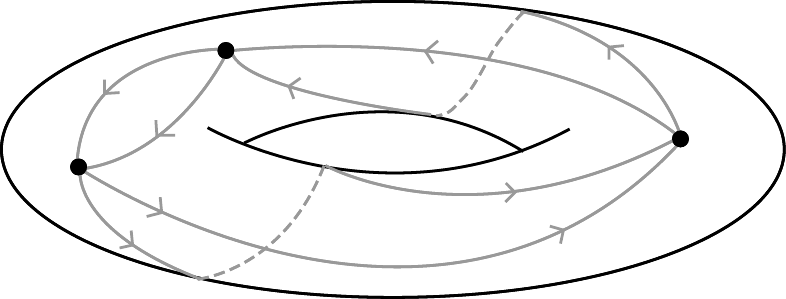}
\put(8,25){$2$}
\put(20,24){$5$}
\put(50,9){$6$}
\put(80,9){$3$}
\put(84,27){$1$}
\put(50,33){$4$}
\end{overpic}
\end{center}
\caption{\label{fig:torus_diagram} A knot diagram on a torus which produces the PD-code in Example \ref{ex:torus_example} . The vertex colorings have been omitted for clarity of the figure, but the sign of each crossing could be deduced from the PD-code.}
\end{figure}

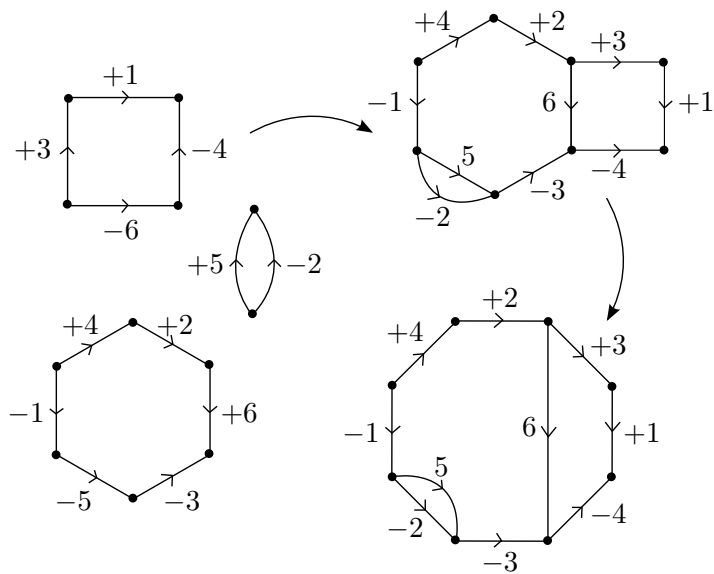
\begin{figure}
\begin{center}
\begin{overpic}[width=9cm]{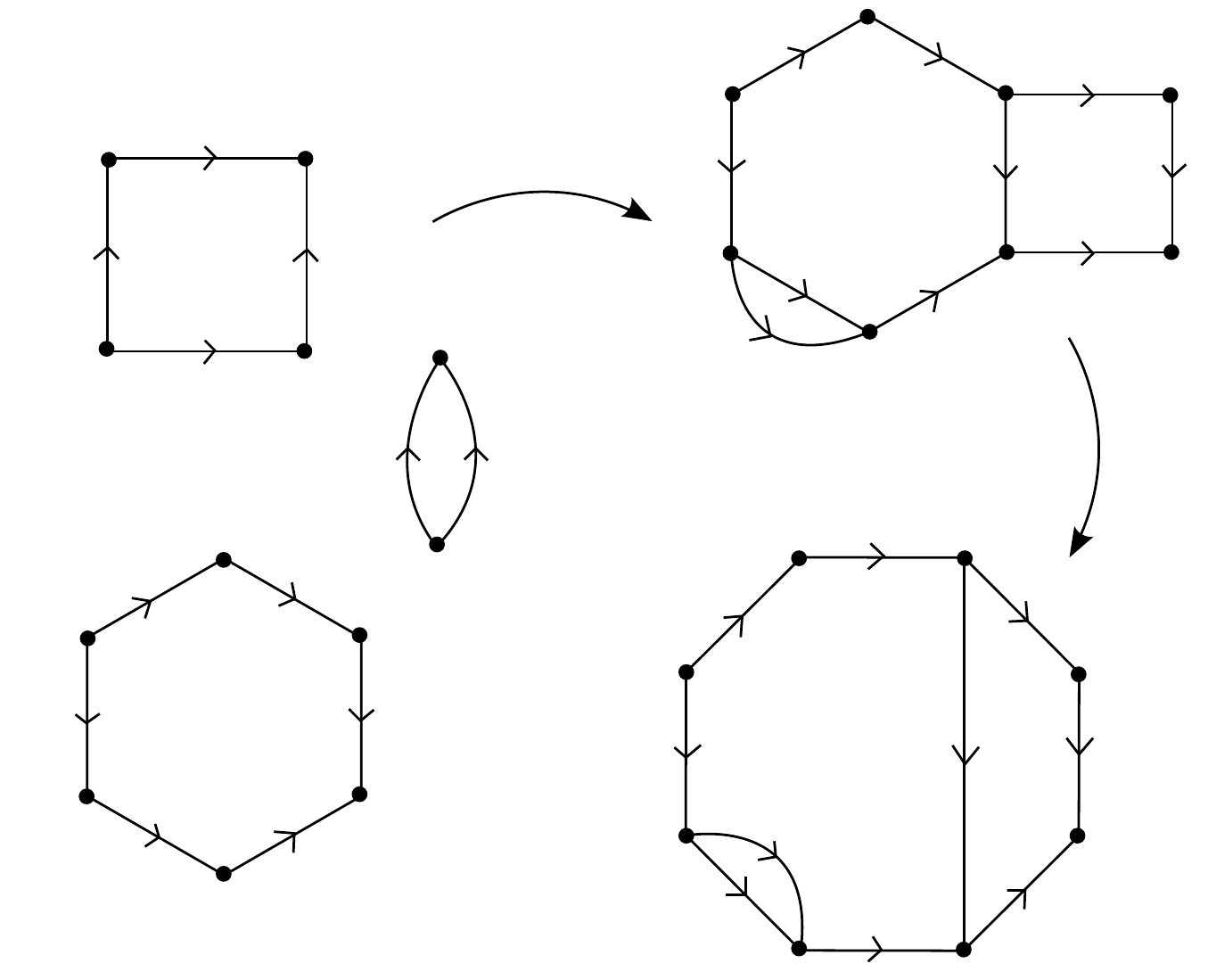}
\put(8,33){$+4$}
\put(22,33){$+2$}
\put(0,20){$-1$}
\put(31.5,20){$+6$}
\put(7,7){$-5$}
\put(23,7){$-3$}

\put(26.5,42){$+5$}
\put(41,42){$-2$}

\put(14,47){$-6$}
\put(14,70){$+1$}
\put(1,59){$+3$}
\put(27,59){$-4$}

\put(53,66){$-1$}
\put(60,78){$+4$}
\put(77,78){$+2$}
\put(86,75){$+3$}
\put(86,56){$-4$}
\put(99,66){$+1$}
\put(79,66){$6$}
\put(77,53){$-3$}
\put(67,58){$5$}
\put(60,49){$-2$}

\put(70,37){$+2$}
\put(86,30){$+3$}
\put(91,17){$+1$}
\put(86,5){$-4$}
\put(70,-2){$-3$}
\put(63,12){$5$}
\put(56,3){$-2$}
\put(49.5,17){$-1$}
\put(56,32){$+4$}
\put(76,18){$6$}
\end{overpic}
\end{center}
\caption{\label{fig:torus_example} The surface of identification obtained by the PD-code in Example \ref{ex:torus_example} . We know this surface is orientable as removing edges $5$ and $6$ from the third frame leaves a single polygon whose boundary is $0$. By computing $\chi = 0$ we conclude that the surface is indeed a torus as claimed.}
\end{figure}
\end{example}

\section[Reidemeister Moves]{Reidemeister Moves}\label{sect:RMoves}

The goal of this section is to establish that $\PD$ modulo a certain set of combinatorial moves is equivalent to the set of link diagrams modulo the Reidemeister Moves, and in turn, to the isotopy classes of links by Reidemeister's Theorem. 

Polyak \cite{ORMoves} established that the four oriented Reidemeister Moves shown in Figure \ref{fig:ORMoves} are sufficient to generate oriented link equivalence on diagrams. 

\begin{figure}
	\begin{overpic}[width=16cm]{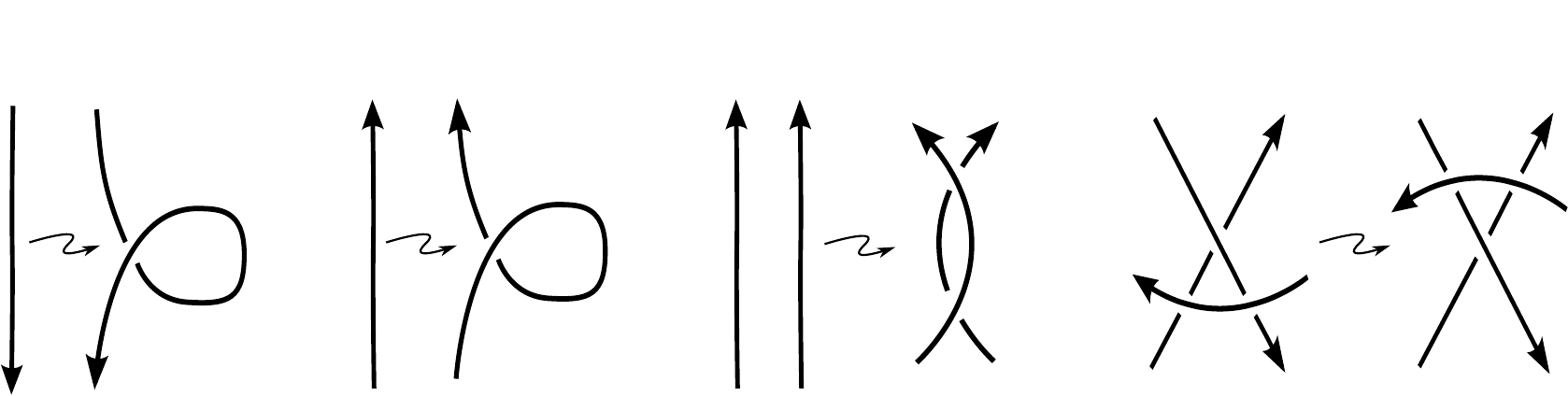}
	\put(1.5,15){$i$}

	\put(7.5,15){$i$}
	\put(12,13){$\alpha$}
	\put(7.5,2){$\beta$}

	\put(25,2){$i$}

	\put(31,15){$\beta$}
	\put(35,13.2){$\alpha$}
	\put(30.2,2){$i$}

	\put(44.5,2){$i$}
	\put(52.5,2){$j$}

	\put(63,9){$\alpha$}
	\put(58.1,9){$\gamma$}
	\put(57,2){$i$}
	\put(64.3,2){$j$}
	\put(64,15){$\delta$}
	\put(56.5,15){$\beta$}

	\put(77,3){$i$}
	\put(74.8,8){$j$}
	\put(79.2,8){$k$}
	
	\put(94,15){$i$}
	\put(91.5,10.5){$k$}
	\put(96.4,10.5){$j$}

	\put(2,22){$1(a)$}
	\put(25,22){$1(a)$}
	\put(54.5,22){$2$}
	\put(86,22){$3$}

%

	\end{overpic}
	\caption{\label{fig:ORMoves} A generating set of oriented Reidemeister Moves. The analogous moves for PD-codes are given in Definition \ref{def:PDMoves}.}
\end{figure}
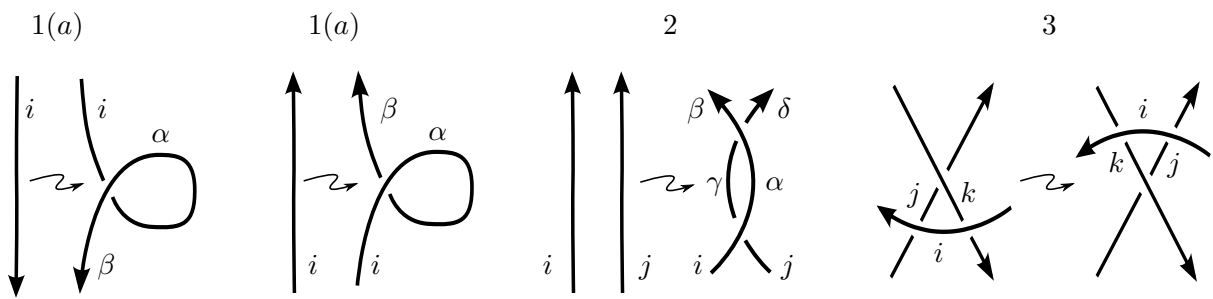

\begin{definition}{\label{def:PDMoves}}
	Given a PD-code $C$ the following moves are called the PD-Moves. Each of these moves corresponds to an oriented Reidemeister Moves shown in Figure \ref{fig:ORMoves} .
	\begin{enumerate}
		\item \phantom{~}
		\begin{enumerate}
			\item Insert or remove a quadruple of the form $[+i, -\beta, -\alpha, +\alpha]$ into $P$ where $\alpha$ and $\beta$ are new labels with $i<\alpha<\beta<i+1$ and $i$ is any label of $P$.
			
			\item Insert or remove a quadruple of the form $[+\alpha, -\alpha, -\beta, +i]$ into $P$ where $\alpha$ and $\beta$ are new labels with $i<\alpha<\beta<i+1$ and $i$ is any label of $P$.

		\end{enumerate}
		
		\item \phantom{~}
		\begin{enumerate}
			\item Let $i$ and $j$ be labels of $P$ both bounding the same $2$-cell. In addition, one of them agrees with the orientation of the $2$-cell and the other does not. Insert the two quadruples $[+j, -\alpha, -\gamma. +i]$ and $[+\gamma , +\alpha, -\delta, -\beta]$ into $P$ where $\alpha$, $\beta$, $\gamma$, and $\delta$ are new labels with $i<\alpha<\beta<i+1$ and $j<\gamma <\delta<j+1$. Undoing this move amounts to removing a pair of quadruples of the form$\{[+j, -\alpha, -\gamma, +i],[+\gamma , +\alpha, -\delta, -\beta]\}$ and make the following identification of the labels. $$i=\alpha=\beta \text{ and } j=\gamma=\delta$$ 
		\end{enumerate}
		\item If $(i, j, k)$ is an orbit of the successor map, then a triple of quadruples of the form 
			$$\{[+(j-1), +i, -j, -(i+1)],[+j, -k, -(j+1), +(k-1)],[+k, -i, -(k+1),+(i-1)]\}$$
			may be replaced by the triple of quadruples
			$$\{[+(j-1), -(k+1), -j, +k],[+(k-1), -(i+1), -k, +i],[+j, +(i-1), -(j+1), -i]\}$$
			or vice versa.
%
	\end{enumerate}
	
\end{definition}

It is a straight forward exercise to verify that the moves described in Definition \ref{def:PDMoves} correspond with the Reidemeister Moves in Figure \ref{fig:ORMoves} . This, the following lemma is immediate and formalizes the relationship between the PD-Moves and the Reidemeister Moves.

\begin{lemma}{\label{lemma:PDMoves}}
	If $C$ and $C'$ are PD-codes that are related by the single PD-Move $p$, then $L(C)$ and $L(C')$ are related by the Reidemeister move $L(p)$. Similarly, if two link diagrams are related by the single Reidemeister Move $r$, then their PD-codes are related by $PD(r)$.
\end{lemma}

We will overload the maps $L: \PD \rightarrow \D$ and $PD: \D \rightarrow \PD$ so that if $p$ is a PD-Move, then $L(p)$ is the corresponding Reidemeister Move. Similarly, if $r$ is a Reidemeister Move, then $PD(r)$ is the corresponding PD-Move. The equivalence relation generated by the PD-Moves will be called $P$ and the equivalence relation generated by the Reidemeister Moves will be called $R$.

The following proposition gives a convenient way to describe the relationship between link diagrams modulo the Reidemeister Moves and PD-codes module the PD-Moves.

\begin{proposition}
	The following diagram commutes.

\begin{center}
\tikzset{node distance=2cm, auto}
\begin{tikzpicture}
  \node (PD) {$\PD$};
  \node (L) [right of=PD] {$\D$};
  \node (pd) [below of=PD] {$\PD / P$};
  \node (l) [right of=pd] {$L / R$};
  \draw[transform canvas={yshift=0.3ex},->] (PD) to node {$L$} (L);
\draw[transform canvas={yshift=-0.3ex},<-] (PD) to node [swap]{$PD$} (L);
  \draw[transform canvas={yshift=-0.3ex},->] (l) to node {$\pi(PD)$} (pd);
  \draw[transform canvas={yshift=0.3ex},<-] (l) to node [swap]{$\pi(L)$} (pd);
  \draw[->] (PD) to node [swap] {$\pi_P$} (pd);
  \draw[->] (L) to node {$\pi_R$} (l);
\end{tikzpicture}
\end{center}

Where the map $\pi(L)$ is defined by choosing a representative $C$ from an equivalence class and setting $\pi(L)(\bar{C}) = \pi_R(L(C))$. The map $\pi(PD)$ is similarly defined by $\pi(PD)(\bar{D}) = \pi_P(PD(D))$.

\end{proposition}

\begin{proof}
	The proof will essentially amount to using repeated applications of Lemma \ref{lemma:PDMoves} . To verify that $\pi(L)$ is well defined we first let $C$ and $C'$ be to representatives of the same equivalence class in $\PD / P$. As such, there is a sequence of PD-Moves $\{p_1, \ldots, p_n\}$ from $C$ to $C'$. Inducting on lemma \ref{lemma:PDMoves} gives that $\{L(p_1), \ldots, L(p_n)\}$ is a sequence of Reidemeister Moves from $L(C)$ to $L(C')$. Thus, $L(C)$ and $L(C')$ are in the same equivalence class of $\D / R$. The argument is essentially identical to show that $\pi(PD)$ is well-defined. 

	\comment{We now turn to showing commutativity. Let $C \in \PD$, it must be shown that $\pi(L)(\pi_P(C)) = \pi_R(L(C))$. } The diagram commutes by construction since the bijection between $\PD$ and $\D$ is used to to define the maps $\pi(L)$ and $\pi(PD)$.
\end{proof}

The main content of this proposition is that the set PD-codes under the PD-Moves is equivalent to link diagrams under the Reidemeister moves as claimed.

\section{\comment{The Action of $\Gamma_\mu$ on PD-codes}  Intrinsic Symmetry of Spherical PD-codes}\label{sect:PDAct}

For this section we will consider only genus $0$ PD-codes, though the action described here after is well-defined for higher genus PD-codes. We can only produce a well-defined link in $S^3$ for genus $0$ PD-codes (without extra information about the embedding of the surface), and hence there is no natural correspondence between the action on the PD-codes and an action on an underlying link. 

The intrinsic symmetry group of a link was first defined by Whitten \cite{Whit} and discussed in detail by Cantarella, et al. \cite{VIGRE}. These symmetries are the generalization of invertiblility and chirality and are described by the group given in the following definition. 

\begin{definition} 
Consider the homomorphism given by
\begin{equation*}
\omega:S_\mu\longmapsto~\operatorname{Aut}(\mathbf{Z}_2^{\mu+1}),\hspace{20pt}p\longmapsto\omega(p)
\end{equation*}
where $\omega(p)$ is defined as
\begin{equation*}
\omega(p)(\epsilon_0,\epsilon_1,\epsilon_2...\epsilon_\mu)=(\epsilon_0,\epsilon_{p(1)},\epsilon_{p(2)}...\epsilon_{p(\mu)}).
\end{equation*}
\\For $\gamma=(\epsilon_0,\epsilon_1,...\epsilon_\mu, p),$ and $\gamma'=(\epsilon'_0,\epsilon'_1,...\epsilon'_\mu, q)\in \mathbf{Z}_2^{\mu+1}\rtimes_\omega S_\mu$, we define the \emph{Whitten group} $\Gamma_\mu$ as the semidirect product $\Gamma_\mu=\mathbf{Z}_2^{\mu+1}\rtimes_\omega S_\mu$
with the group operation
\begin{align*}
\gamma\ast\gamma'&=\left(\epsilon_0,\epsilon_1,\epsilon_2...\epsilon_\mu,p\right)\ast (\epsilon'_0,\epsilon'_1,\epsilon'_2...\epsilon'_\mu,q)\\
& =((\epsilon_0,\epsilon_1,\epsilon_2...\epsilon_\mu)\cdot\omega(p)(\epsilon'_0,\epsilon'_1,\epsilon'_2...\epsilon'_\mu),qp)\\
&=(\epsilon_0\epsilon'_0,\epsilon_1\epsilon'_{p(1)},\epsilon_2\epsilon'_{p(2)}...\epsilon_\mu\epsilon'_{p(\mu)},qp)
\end{align*}
We may also use the notation $\Gamma(L)$ to refer to the \textbf{Whitten group} $\Gamma_\mu$.
\label{def:whittengroup}
\end{definition}

The action of this group on link diagrams is also described by Cantarella, et al. \cite{VIGRE}. Our present goal is to describe the corresponding action on PD-codes. The main result of this section is that one can choose a preferred PD-code for each link type so that no non-trivial PD-code is fixed by a non-identity element of the Whitten group, i.e, the action is free.

\begin{definition}\label{def:PDAct}
Consider an element $\gamma = (\epsilon_0,\epsilon_1,\ldots,\epsilon_\mu,p) \in \Gamma_\mu$. We define an action of $\gamma$ on the labels of a PD-code  by the following operations (which must be applied in order):

\begin{enumerate}
\item First, we apply $p$ to the first component of each label.
\item If $\epsilon_0 = 1$ we do nothing. If $\epsilon_0 = -1$ we shift each positive crossing to the right by $1$ and each negative crossing to the left by $1$.
\item If $\epsilon_i = 1$ we do nothing. If $\epsilon_i = -1$ we first apply the permutation 

$$\left( 1 \right) \left( 2 \hspace{0.5cm} n_1 \right) \left( 3 \hspace{0.5cm} n_1-1 \right) \ldots \left( \frac{n+2}{2}-1 \hspace{0.5cm} \frac{n+2}{2}+1 \right) \left( \frac{n_1+2}{2} \right)$$

 to the second component of the labels whose first component label is $i$, ignoring signs, then shift each quadruple to the right (or left) by $2$ if the $i$th component is the under crossing, and finally switch the sign on the second component of every label with first component $i$.

\end{enumerate}
\end{definition}

%
%

It is clear that the action of these elements produces a set of quadruples, but we will now show that in fact there is a well-defined action on PD-codes.

\begin{proposition}\label{prop:PDAct}
The previous definition gives a well-defined action of $\Gamma_\mu$ on PD-codes. 
\end{proposition}

\begin{proof}
We must first show that the resulting collection of labels satisfies properties $1-4$ of Definition~\ref{def:PDSet} and thus produces a valid PD-code. Property $1$ persists as the only change to the sign of a label changes the sign of all labels, thus there are still exactly one positive label and one negative label for each edge. Property $2$ persists as positive crossings have a positive second label in the first and last slots, and after shifting to the right we again begin with a positive crossing. Similarly, negative crossing have positive labels in the first two positions and so shifting to the left also gives a quadruple beginning with a positive label. Property $3$ also checks out again because of the global sign change. To see that property $4$ is preserved we first note that since the sign change affects all labels we will still have that non-adjacent labels have opposite sign. However, the sign change does affect the lesser edge is positive condition. Luckily, the permutation $(1)(2 \hspace{0.5cm} n_1)(3 \hspace{0.5cm} n_1-1)\ldots(\frac{n_1+2}{2})$ straightens this out.

We now turn the task of showing that we have a group action. Applying the permutations to the first coordinate of each label is the natural action of the symmetric group on the integers $\{1,\ldots,\mu\}$. The information contained in a quadruple of a PD-code represents an under-crossing and an over-crossing along with orientation information for both. This structure admits a $\Z_2 \cross (\Z_2)^\mu$ action by switching the overcrossing for the undercrossing and reversing the orientations. In the case that both the undercrossing and the overcrossing appear in the same component we have the diagonal action where the second two group elements of $\Z_2 \cross (\Z_2 \cross \Z_2)$ are the same. The effect of this action on the labels in the PD-code is described by $2$ and $3$ of Definition~\ref{def:PDAct} as seen in Figure~\ref{fig:GamAction}.

\begin{figure}
\begin{center}
\begin{overpic}[height=7cm]{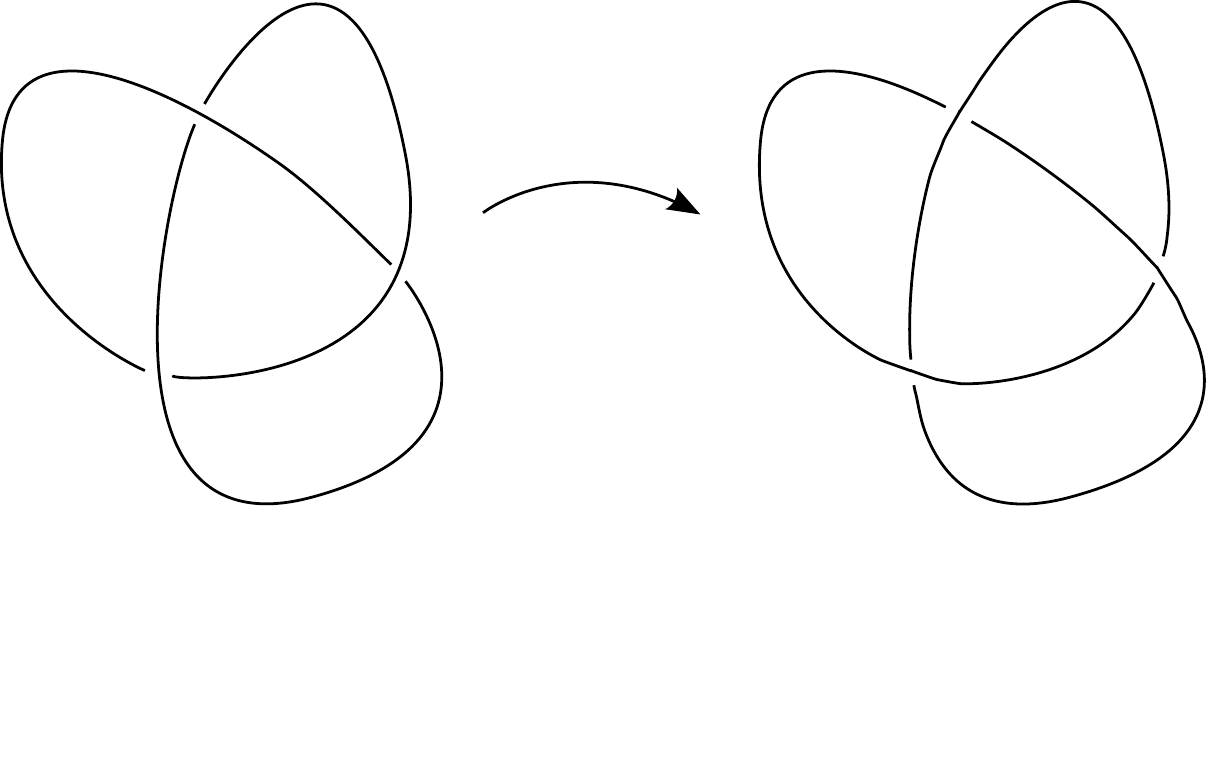}
\put(34,53){$1$}
\put(28,33){$2$}
\put(2,59.5){$3$}
\put(26,50){$4$}
\put(28,20){$5$}
\put(11,46){$6$}
\put(7,13){$\{ [+4,-2,-5,+1],$}
\put(9,8){$[+2,-6,-3,+5],$}
\put(9,3){$[+6,-4,-1,+3]\}$}

\put(97,53){$1$}
\put(91,33){$6$}
\put(64,59.5){$5$}
\put(89,50){$4$}
\put(91,20){$3$}
\put(74,46){$2$}
\put(70,13){$\{ [+6,+3,-1,-4],$}
\put(72,8){$[+2,+5,-3,-6],$}
\put(72,3){$[+4,+1,-5,-2]\}$}

\put(41,51){$(-1,-1)$}

\end{overpic}
\end{center}
\caption{\label{fig:GamAction} The action of $(-1,-1)$ on a diagram of the trefoil. Note that the affect on the PD-codes is as described in Definition~\ref{def:PDAct}.}
\end{figure}
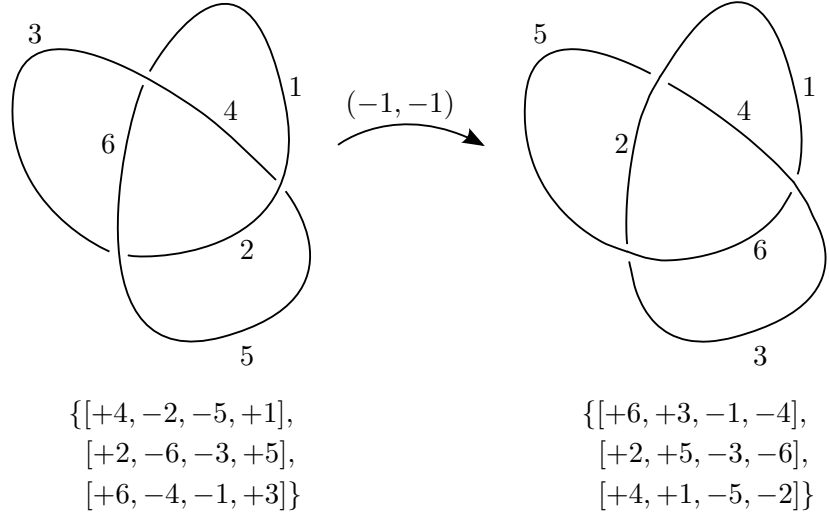


\end{proof}

We now address the issue of the existence of PD-codes which are fixed under the action of an element of $\Gamma_\mu$. While these do exist, we will show that we can produce a PD-code for each link that is not fixed by a non-trivial element of $\Gamma$. In particular, if a PD-code is fixed by some non-trivial element in $\Gamma_\mu$, then we can add a Reidemeister $1$ loop to it so that it still represents the same knot type but is not fixed by any non-identity element of $\Gamma_\mu$. In this way we can produce a collection of PD-codes on which the action of $\Gamma$ is free. \\

\begin{example}
Consider the diagram and PD-code of the Hopf Link shown in Figure~\ref{fig:PDFixed}. The PD-code associated to this diagram is 

$$\{[(1,+2),(2,-2),(1,-1),(2,+1)],[(2,+2),(1,-2),(2,-1),(1,+1)]\}.$$ 

If we act on this PD-code by $\gamma = (1,1,(12))\in \Gamma_2$ the result is the PD-code 
$$\{[(2,+2),(1,-2),(2,-1),(1,+1)],[(1,+2),(2,-2),(1,-1),(2,+1)]\},$$

thus the PD-code is fixed by the action of $\gamma$.
\end{example}

\begin{figure}
\begin{center}
\begin{overpic}[height=6cm]{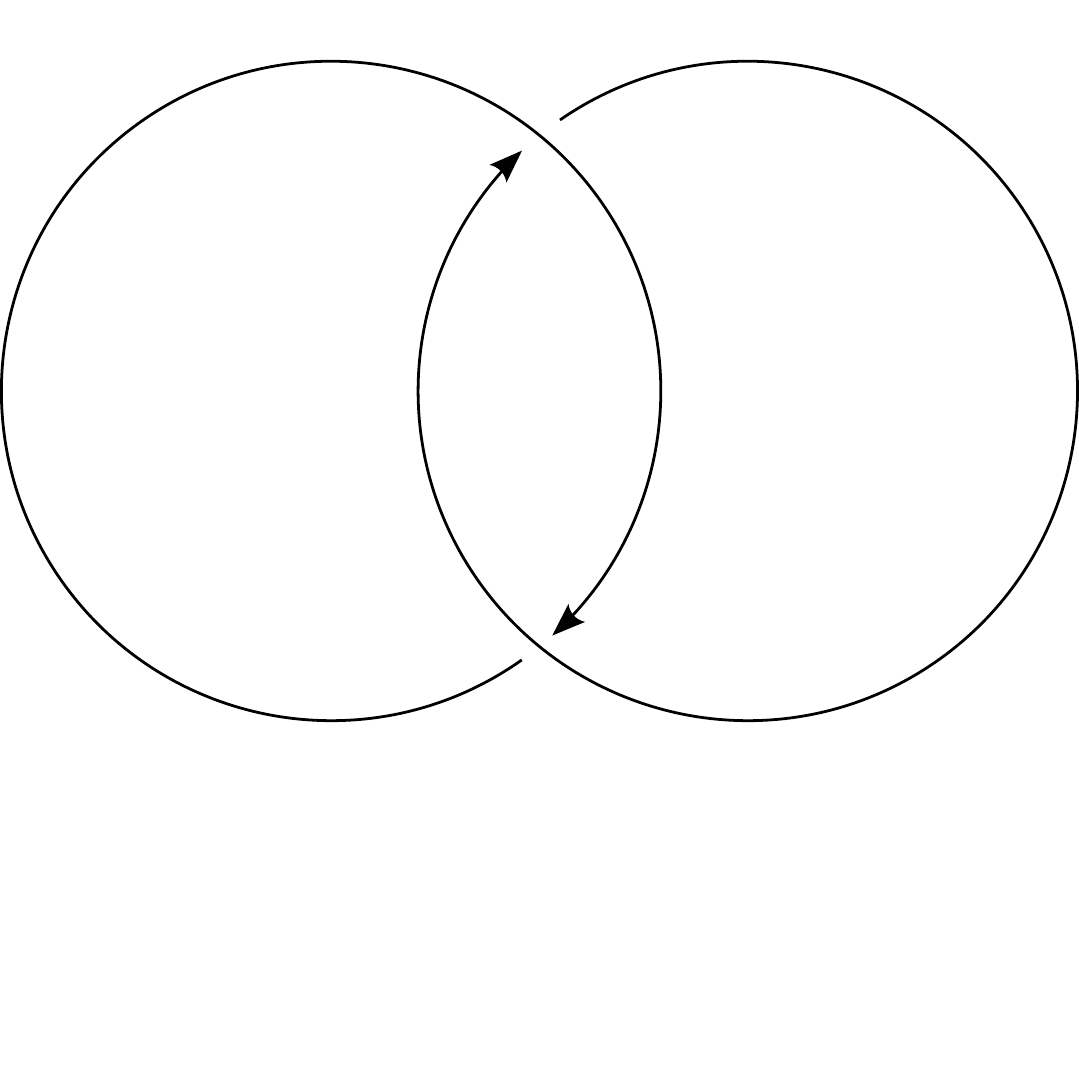}
\put(5,20){$\{[(1,+2),(2,-2),(1,-1),(2,+1)],$}
\put(8.5,11){$[(2,+2),(1,-2),(2,-1),(1,+1)]\}$}
\put(-1,90){$(1,1)$}
\put(60,72){$(1,2)$}
\put(88,90){$(2,1)$}
\put(25,72){$(2,2)$}
\end{overpic}
\end{center}
\caption{\label{fig:PDFixed} An example of a PD-code that is fixed by $(1,1,1,(12))$.}
\end{figure}

\begin{lemma}
In every equivalence class of PD-codes modulo the PD-moves there is a representative which is fixed only by the identity element in $\Gamma_\mu$.
\end{lemma}

\begin{proof}
We first observe that a Reidemeister $1$ loop occurs in a diagram if and only if the corresponding PD-code contains a quadruple with two pairs with the same integer appearing as second components with opposite signs. \\

Suppose we have a PD-code for a knot with exactly one Reidemeister $1$ loop. Then there is a quadruple of the form $[i,-i,j,k]$,$[i,k,j,-j]$,$[i,j,-j,k]$, or $[i,j,k,-i]$. But, none of these are fixed by a shift by $1$. Thus, no PD-code containing a Reidemeister $1$ loop can be fixed by the action of $(-1,1)$. Similarly, since there is no other quadruple contained both positive and negative $i$ (or $j$) this PD-code cannot be fixed by $(1,-1)$. 

Thus, if we have a PD-code for a knot that is fixed by some element of $\Gamma$ we may modify it by first removing all Reidemeister $1$ loops and then adding a single Reidemeister $1$ loop. Given a table of PD-codes for knots we simply traverse the list and fix each code to produce preferred list.

To generalize to links we simple remove all Reidemeister $1$ loops and then add $n$ loops back to component $n$ for each component $1,\ldots,\mu$. 
\end{proof}

\begin{corollary}
There is a preferred table of link diagrams on which the action of $\Gamma$ is free.
\end{corollary}

The preferred list of diagrams from the previous corollary will play a central role in the enhanced prime decomposition theorem for knots \cite{Ma1}.

%
%
%

\begin{figure}
\begin{center}
\begin{overpic}[height=2cm]{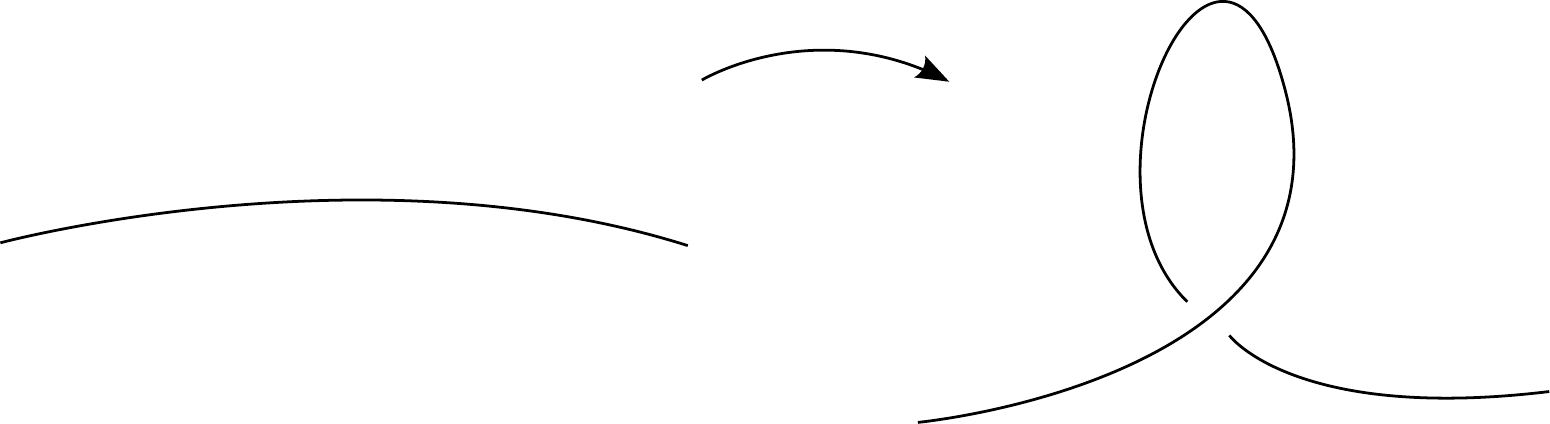}
\put(20,16){$n$}

\put(63,4){$n$}
\put(85,21){$n+1$}
\put(91,4){$n+2$}

\end{overpic}
\end{center}
\caption{\label{fig:Reid1} Adding a Reidemeister $1$ loop to arc $n$.}
\end{figure}
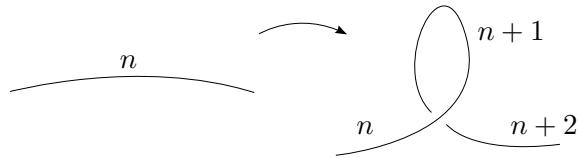

\section[Future Directions]{Future Directions}\label{sect:Future}

We have in many places restricted our attention to the PD-codes whose associated surfaces are spheres because these correspond to classical knots. Relaxing this condition gives rise to the collection of PD-codes whose associated surfaces are orientable, but of arbitrary genus. This set of PD-codes includes the virtual links~\cite{VKnots2} and certainly warrant future study. Another obvious generalization is embed our surfaces in $3$-manifolds other than $S^3$ and develop a combinatorial theory for links in arbitrary $3$-manifolds. 


Sometimes intrinsic symmetries arise as combinatorial automorphisms of the associated cell complex. Such symmetries could be called \textbf{diagrammatic symmetries} and are convenient for finding intrinsic symmetries of links. These methods were used by Cantarella, et al. \cite{VIGRE}, but formal attention should be given to this class of symmetries.

\bibliographystyle{alpha}
\bibliography{refs}

\end{document}